\documentclass[12pt,reqno]{amsart}

\usepackage{amsmath}
\usepackage{amsthm}
\usepackage{amssymb}
\usepackage{amsfonts}
\usepackage{latexsym}
\usepackage{hyperref}
\usepackage{cite}
\usepackage{xcolor}

\usepackage{todonotes}
\usepackage{comment}
\usepackage{multirow}
\usepackage{fullpage}

\newcommand{\Q}{\mathbb{Q}}

\newcommand{\K}{\mathbb{K}}

\renewcommand{\L}{\mathbb{L}}

\newcommand{\PP}{\mathfrak{p}}

\newcommand{\Gal}{\operatorname{Gal}}

\newcommand{\imod}[1]{\,\left(\textnormal{mod }#1\right)\,}

\makeatletter

\newcommand{\red}[1]{{\color{red} #1}}

\makeatother

\newtheorem{theorem}{Theorem}[section]
\newtheorem{corollary}[theorem]{Corollary}
\newtheorem{lemma}[theorem]{Lemma}

\theoremstyle{definition}

\newtheorem*{remark}{Remark}

\title{On the constants in Mertens' theorems for primes in arithmetic progressions}

\author[D.~Keliher]{Daniel~Keliher}
\address{University of Georgia, Department of Mathematics, Athens, GA, USA} 
\email{keliher@uga.edu}

\author[E.~S.~Lee]{Ethan~Simpson~Lee}
\address{University of Bristol, School of Mathematics, Fry Building, Woodland Road, Bristol, BS8 1UG} 
\email{ethan.lee@bristol.ac.uk}
\urladdr{\url{https://sites.google.com/view/ethansleemath/home}}

\begin{document}

\maketitle

\begin{abstract}
A 1976 result from Norton may be used to give an asymptotic (but not explicit) description of the constant in Mertens' second theorem for primes in arithmetic progressions. Assuming the Generalized Riemann Hypothesis, we give an effective description of Norton's observation.
\end{abstract}

\section{Introduction}

Suppose that $p$ denotes prime numbers, $B = 0.2614\ldots$ is the Meissel--Mertens constant, and $\gamma = 0.5772\ldots$ is the Euler--Mascheroni constant. Mertens famously proved the following three results \cite{mertens1874ein}, which are collectively called Mertens' theorems:
\begin{align}
    \sum_{p\leq x} \frac{\log{p}}{p} &= \log{x} + O(1), \label{eqn:M1}\\
    \sum_{p \leq x} \frac{1}{p} &= \log\log{x} + B + O\bigg(\frac{1}{\log x} \bigg), \label{eqn:M2}\\
    \prod_{p\leq x} \bigg(1 - \frac{1}{p} \bigg) &= \frac{e^{-\gamma}}{\log x} \big(1+o(1) \big). \label{eqn:product_formula}
\end{align}
Mertens' theorems have been generalized into several settings. For example, Williams \cite{Williams} generalized them for primes in arithmetic progressions, Rosen \cite{rosen1999generalization} generalized them for the number field setting, Arango-Pi\~{n}eros, Keyes, and the first author \cite{M4CS} generalized them for prime ideals in a conjugacy class of the Galois group of a Galois extension of number fields, Yashiro \cite{M4SC} proved them for the Selberg class, and Lebacque \cite{Lebacque} generalized them for smooth, absolutely irreducible, projective algebraic varieties. Hasegawa, Saito, and Sato \cite{MertensGraphTheoryGen_Abelian, MertensGraphTheoryGen_Galois} also generalized Williams' result for graph covers.


Any analogue of \eqref{eqn:M2} will involve a constant that can be computed exactly when the setting is fixed. However, when the setting can vary, the constant could fall anywhere inside a range, where the upper and lower bounds will depend upon certain invariants in that setting. For example, suppose that $\K$ is a number field of degree $n_{\K} \geq 2$ with discriminant $\Delta_{\K}$, $\PP$ are the prime ideals of $\K$, $N(\PP)$ is the norm of $\PP$, and $\kappa_{\K}$ is the residue at $s=1$ of the Dedekind zeta-function $\zeta_{\K}(s)$ associated to $\K$. Garcia and the second author proved in \cite{GarciaLeeRamanujan} that
\begin{equation*}
    \sum_{N(\PP) \leq x} \frac{1}{N(\PP)} = \log\log{x} + M_{\K} + O\bigg(\frac{1}{\log x} \bigg) ,
\end{equation*}
in which the implied constant in the big-$O$ depends only on $n_{\K}$ and $\Delta_{\K}$, and
\begin{equation*}
    \gamma + \log{\kappa_{\K}} - n_{\K} 
    \leq M_{\K} 
    \leq \gamma + \log{\kappa_{\K}} .
\end{equation*}


In the present paper, we focus on generalisations of Mertens' theorems for primes in arithmetic progressions. These have received a lot of recent attention. First, Williams proved (in \cite[Theorem 1]{Williams}) that if $q$ and $\ell$ are integers such that $1\leq\ell\leq q$ and $(\ell,q) = 1$, then there exists a constant $C(q,\ell)$ such that
\begin{equation}\label{eqn:MertensThirdPrimesAP}
    \prod_{\substack{p\leq x\\ p\,\equiv\,\ell\imod{q}}} \left(1 - \frac{1}{p}\right)
    = \frac{C(q,\ell)}{(\log{x})^{1/\varphi(q)}} + O\!\left(\frac{1}{(\log{x})^{1+1/\varphi(q)}}\right) ,
\end{equation}
where the implied constant depends on $q$, and $\varphi$ denotes the Euler totient function; this generalizes \eqref{eqn:product_formula} for primes in arithmetic progressions. Williams did give a description of the constant $C(q,\ell)$, but his definition depends on a Dirichlet series $K(s,\chi)$ that is somewhat cumbersome to define. In \cite{LanguascoZaccagnini}, Languasco and Zaccagnini gave the following alternate description:
\begin{equation}\label{eq:C_LZ}
\begin{split}
    C(q,\ell)^{\varphi(q)} 
    &= e^{-\gamma} \prod_p \left(1 - \frac{1}{p}\right)^{\alpha(p;q,\ell)} , \\
    &\qquad\text{where}\quad
    \alpha(p;q,\ell) = 
    \begin{cases}
        \varphi(q) - 1 &\text{if $p\equiv \ell\imod{q}$,}\\
        - 1 &\text{if $p\not\equiv \ell\imod{q}$.}
    \end{cases}
\end{split}
\end{equation}
In \cite{LanguascoZaccagniniNumerical, LZ_identities}, Languasco and Zaccagnini also demonstrated how to compute $C(q,\ell)$ with high accuracy for a broad range of $q$ and $\ell$  and gave a simpler proof of \eqref{eq:C_LZ}.  Furthermore, in \cite{LZ2008}, Languasco and Zaccagnini studied average values of $C(q, \ell)$, wherein they established analogues of theorems from Bombieri--Vinogradov and Barban--Davenport--Halberstam.

Next, one can use \eqref{eqn:MertensThirdPrimesAP} to generalize \eqref{eqn:M2} for primes in arithmetic progressions. That is, if $q$ and $\ell$ are integers such that $1\leq\ell\leq q$ and $(\ell,q) = 1$, then
\begin{equation}\label{eqn:MertensSecondAPs}
    \sum_{\substack{p\leq x\\ p\,\equiv\,\ell\imod{q}}} \frac{1}{p}
    = \frac{\log\log{x}}{\varphi(q)} + M(q,\ell) + O\!\left(\frac{1}{\log{x}}\right),
\end{equation}
where the implied constant depends on $q$ and $M(q,\ell) = - \log{C(q,\ell)} + B(q,\ell)$ such that
\begin{equation*}
    B(q,\ell) = \sum_{p\,\equiv\,\ell\imod{q}} \left\{\frac{1}{p} + \log\left(1 - \frac{1}{p}\right)\right\} .
\end{equation*}
Once \eqref{eqn:MertensSecondAPs} is known, it is straightforward to generalize \eqref{eqn:M1} into the context of primes in arithmetic progressions using partial summation.

Pomerance proved that $M(q,\ell)$ is bounded as $q$ and $\ell$ vary in \cite{Pomerance}. Languasco and Zaccagnini also explored the constant $M(q,\ell)$ in \cite{LanguascoZaccComps}. As part of this work, they used their extensive data to conjecture that $M(q,\ell) \sim \delta_{\ell}/\ell$, where $\delta_{\ell} = 1$ if $\ell$ is prime and $\delta_{\ell} = 0$ otherwise. It turned out that Norton had already proved a lemma in 1976 (see \cite[Lemma 6.3]{Norton}), which is a more general version of this relationship, so his contribution was included as an appendix to Languasco and Zaccagnini's paper. In particular, a special case of Norton's lemma tells us that
\begin{equation}\label{eqn:Norton}
    M(q,\ell) = \frac{\delta_{\ell}}{\ell} + O\!\left(\frac{\log{q}}{\varphi(q)}\right) ,
\end{equation}
where the implied constant in \eqref{eqn:Norton} was not specified or described. Note that we will see $\varphi(q) \geq \sqrt{q}$ for $q > 3$ such that $q\neq 6$ in Lemma \ref{lem:varphi_lower_bounds}. It follows that \eqref{eqn:Norton} can be re-written for these $q$ as
\begin{equation}\label{eqn:Norton_rewritten}
    M(q,\ell) = \frac{\delta_{\ell}}{\ell} + O\!\left(\frac{\log{q}}{\sqrt{q}}\right) ,
\end{equation}
although the implied constant in \eqref{eqn:Norton_rewritten} remains unspecified. 

Our objective in this paper is to describe the implied constant in \eqref{eqn:Norton_rewritten}, assuming the Generalized Riemann Hypothesis (GRH) is true. We need to assume the GRH because the unconditional knowledge we have on the distribution of primes in arithmetic progressions is not good enough to match the asymptotic order of the big-$O$ in \eqref{eqn:Norton_rewritten} without assuming it. The following result is an important milestone toward satisfying the aforementioned objective.



\begin{theorem}\label{thm:explicit_fmla_constant}
If $q\geq 3$ and the GRH is true, then
\begin{equation*}
    \left| M(q,\ell) - \frac{\delta_{\ell}}{\ell} + \frac{\log\log{q}}{\varphi(q)}\right|
    < \tau_1(q) \frac{\varphi(q)^{-1}}{q-1} + \tau_2(q) \frac{\log{q}}{\sqrt{q}} ,
\end{equation*}
in which 
\begin{align*}
    \tau_1(q) &= \red{B + \frac{1}{2 \log{q}} + \frac{1}{(\log{q})^2}} 
    \quad\text{and}\\
    \tau_2(q) &= \frac{15 + \frac{84.89}{q^{2/3}}}{8\pi} + \frac{6 + \frac{15}{4\pi} + \frac{3}{4\pi q^{2/3}}}{\log{q}} + \frac{22.29 + q^{- 2/3}}{(\log{q})^2} + \frac{20.58}{(\log{q})^3} .
\end{align*}
\end{theorem}

\begin{table}[]
    \centering
    \begin{tabular}{c|cc}
        $q$ & $\tau_1(q)$ & $\tau_2(q)$ \\
        \hline
        $3$ & $1.54515$ & $43.26016$ \\
        $4$ & $1.14251$ & $26.72431$ \\
        $5$ & $0.95822$ & $19.94616$ \\
        $10$ & $0.66726$ & $10.40159$ \\
        $15$ & $0.58249$ & $7.92122$ \\
        $20$ & $0.53983$ & $6.73171$ \\
        $25$ & $0.51335$ & $6.01506$ \\
        $50$ & $0.45465$ & $4.49411$ \\
        $100$ & $0.41722$ & $3.58205$
    \end{tabular}
    \caption{Computations of $\tau_1(q)$, $\tau_2(q)$ in Theorem \ref{thm:explicit_fmla_constant} for a selection of $q\geq 3$.}
    \label{tab:comps}
\end{table}

It is important to note that $\tau_1(q)$ and $\tau_2(q)$ in Theorem \ref{thm:explicit_fmla_constant} are decreasing functions in $q$. To demonstrate this, we have provided a collection of computations in Table \ref{tab:comps}. It follows from Theorem \ref{thm:explicit_fmla_constant} and these computations that if $q\geq 3$ and the GRH is true, then we can prove the following result (Corollary \ref{cor:Norton_explicit}), which is clearly an effective version of \eqref{eqn:Norton_rewritten}. If one can consider $q$ in a restricted range, then the constant in Corollary \ref{cor:Norton_explicit} can be refined in a straightforward manner.

\begin{corollary}\label{cor:Norton_explicit}
If $q\geq 3$ and the GRH is true, then
\begin{equation*}
    \left| M(q,\ell) - \frac{\delta_{\ell}}{\ell} \right| \leq \frac{\red{43.95} \log{q}}{\sqrt{q}} .
\end{equation*}
\end{corollary}

We will prove Theorem \ref{thm:explicit_fmla_constant} in Section \ref{sec:main_theorems} in three important steps. First, in \eqref{eqn:Step2}, we will use \eqref{eq:C_LZ} to prove unconditionally that
\begin{equation}\label{eqn:look_at_me}
    \left| M(q,\ell) - \frac{\delta_{\ell}}{\ell} + \frac{\log\log{q}}{\varphi(q)} - \sum_{p>q} \left\{\frac{\beta(p;q,\ell)}{p} + \frac{1}{\varphi(q)} \log\left(1 - \frac{1}{p}\right)\right\} \right| \leq \frac{\varphi(q)^{-1}}{(\log{q})^2} ,
\end{equation}
in which $\beta(p;q,\ell) = 1$ if $p\equiv\ell\imod{q}$ and $\beta(p;q,\ell) = 0$ otherwise. Second, we control the sum over primes $p>q$ in \eqref{eqn:look_at_me} using technical lemmas and the explicit Prime Number Theorems (that depend on the GRH) presented in Section \ref{ssec:EPNT}. Finally, we will use the lower bounds for $\varphi(q)$ that are presented in Section \ref{ssec:varphi_lower_bounds} and computations to deduce Theorem \ref{thm:explicit_fmla_constant}.

In Section \ref{sec:proof_of_theorems}, we state some supporting results and import explicit versions of the Prime Number Theorem that will be useful later. Our main results, namely Theorem \ref{thm:explicit_fmla_constant} and Corollary \ref{cor:Norton_explicit}, will be proved in Section \ref{sec:main_theorems}. 

We suppose throughout that $O^\star$ is big-$O$ notation with implied constant at most one, $\varphi$ denotes the Euler totient function, $\gamma$ is the Euler--Mascheroni constant, and $M$ is the Meissel--Mertens constant. We always reserve $q$ for the modulus of the arithmetic progression under consideration and $p$ for rational primes.

\begin{remark}
It should be possible to extend the techniques we will present in this paper to prove an unconditional result. The argument would proceed along similar lines, but would require explicit and \textit{unconditional} versions of the Prime Number Theorem for primes in arithmetic progressions (e.g. \cite{BennettEtAlPAPS}) and the Prime Number Theorem (e.g. \cite{BroadbentEtAl, Rosser}). There would be extra (non-trivial) technical challenges to account for though. 
\end{remark}

\red{
\begin{remark}
This version reflects the paper that will published in the Proceedings of the Integers Conference 2023, with a few differences that will be marked in red. In particular, we observe a simple improvement in Lemma \ref{lem:Bq}, which leads to a meagre improvement in our main results. Our code to recover the constants in this paper is also available on GitHub: follow \href{https://github.com/EthanSLee/On-the-constants-in-Mertens-theorems-for-primes-in-arithmetic-progressions/blob/main/main.py}{\texttt{this link}}.
\end{remark}
}

\section{Technical Lemmas}\label{sec:proof_of_theorems}

In our proof of Theorem \ref{thm:explicit_fmla_constant}, we will require two technical lemmas which are presented in Section \ref{ssec:TL}. To prove one of these technical lemmas, we need explicit information about
\begin{equation*}
    \theta(x) = \sum_{p\leq x} \log{p} 
    \quad\text{and}\quad 
    \theta(x;q,\ell) = \sum_{\substack{p\leq x\\p\,\equiv\,\ell\imod{q}}} \log{p} .
\end{equation*}
These sums are well-known to be objects of interest in the Prime Number Theorem and the Prime Number Theorem for primes in arithmetic progressions. To this end, we present explicit versions of the error in each of these Prime Number Theorems in Section \ref{ssec:EPNT}. In general, we will assume the Riemann Hypothesis (RH) or Generalised Riemann Hypothesis (GRH), so that we can get the strongest bounds on $M(q,\ell)$ in the end. Finally, we give lower bounds for $\varphi(n)$ in Section \ref{ssec:varphi_lower_bounds} which will assist some computations in Section \ref{sec:main_theorems}.

\subsection{Explicit Versions of the Prime Number Theorem}\label{ssec:EPNT}

Suppose that $\mathbb{K}\subseteq \mathbb{L}$ is a Galois extension of number fields such that $n_{\mathbb{L}} = [\mathbb{L} : \mathbb{Q}]$, $\Delta_{\mathbb{L}}$ is the discriminant of $\mathbb{L}$, $\mathfrak{p}$ is a prime ideal of $\K$ that does not ramify in $\L$, $\sigma_{\mathfrak{p}}$ is the Frobenius class of $\mathfrak{p}$ in $\mathbb{L}/\mathbb{K}$, and $\mathcal{C}$ is any conjugacy class in $\mathcal{G} = \Gal(\L/\K)$. If $\K = \L$, then $\mathfrak{p}$ corresponds to the rational primes $2,3,5,\dots$. Moreover, if $\K = \Q$ and $\L = \Q(\omega_q)$, where $\omega_q$ is the $q$th root of unity, then there are are $\varphi(q)$ conjugacy classes $\mathcal{C}$ and $\mathfrak{p}$ such that $\sigma_{\mathfrak{p}} = \mathcal{C}$ correspond to primes in an arithmetic progression. Denoting the norm of $\mathfrak{p}$ in $\K$ by $N(\mathfrak{p})$, it follows that
\begin{equation*}
    \theta_{\mathcal{C}}(x) 
    = \sum_{\substack{N(\mathfrak{p}) \leq x\\\sigma_{\mathfrak{p}} = \mathcal{C}}} \log{N(\mathfrak{p})}
\end{equation*}
is a broad generalisation of $\theta(x)$ and $\theta(x;q,\ell)$.

Assuming that the GRH for Dedekind zeta-functions is true, Greni\'{e} and Molteni proved that for all $x\geq 2$, we have
\begin{equation}\label{eqn:GM_CDT_psi}
\begin{split}
    \Bigg|\sum_{\substack{N(\mathfrak{p}^r) \leq x\\\sigma_{\mathfrak{p}^r} = \mathcal{C}}} \log{N(\mathfrak{p})} &- \frac{\#\mathcal{C}}{\#\mathcal{G}} x\Bigg| \\
    &\leq \frac{\#\mathcal{C}}{\#\mathcal{G}} \sqrt{x} \left[\left(\frac{(\log{x})^2}{8\pi} + 2\right) n_{\mathbb{L}} + \left(\frac{\log{x}}{2\pi} + 2\right) \log{|\Delta_{\mathbb{L}}|}\right] .
\end{split}
\end{equation}
This is an explicit version of the Chebotar\"{e}v density theorem. Next, the norm is multiplicative and $N(\mathfrak{p}) = p^k$ for some unique rational prime $p$ and $1\leq k\leq n_{\K}$, so it follows from \cite[Theorem~13]{Rosser} that
\begin{equation*}
    0 \leq \sum_{\substack{N(\mathfrak{p}^r) \leq x\\\sigma_{\mathfrak{p}^r} = \mathcal{C}}} \log{N(\mathfrak{p})} - \theta_{\mathcal{C}}(x)
    \leq n_{\K} \sum_{\substack{p^r \leq x\\r\geq 2}} \log{p}
    \leq 1.42620 n_{\K} \sqrt{x} .
\end{equation*}
Therefore, \eqref{eqn:GM_CDT_psi} implies
\begin{equation}\label{eqn:GM_CDT_theta}
\begin{split}
    \Bigg|\theta_{\mathcal{C}}(x) &- \frac{\#\mathcal{C}}{\#\mathcal{G}} x\Bigg| \\
    &\leq \frac{\#\mathcal{C}}{\#\mathcal{G}} \sqrt{x} \left[\left(\frac{(\log{x})^2}{8\pi} + 2\right) n_{\mathbb{L}} + \left(\frac{\log{x}}{2\pi} + 2\right) \log{|\Delta_{\mathbb{L}}|}\right] + 1.43 n_{\K} \sqrt{x}.
\end{split}
\end{equation}
The following result describes the special cases of \eqref{eqn:GM_CDT_theta} that we will need.

\begin{theorem}\label{thm:pnts}
Suppose that the GRH is true. If $(\ell,q)=1$, $x\geq 2$, and $q\geq 3$, then
\begin{equation}\label{eqn:can_be_improved}
    \left|\theta(x;q,a) - \frac{x}{\varphi(q)}\right|
    < \left(\frac{(\log{x})^2}{8\pi} + \left(\frac{\log{x}}{2\pi} + 2\right)\log{q} + 3.43\right) \sqrt{x} 
\end{equation}
and
\begin{equation}\label{eqn:pnts_smallT}
    \left|\theta(x) - x\right|
    < \left(\frac{(\log{x})^2}{8\pi} + 3.43\right) \sqrt{x}  .
\end{equation}
\end{theorem}

\begin{proof}
Insert $\mathbb{L} = \mathbb{Q}(\omega_q)$ and $\K = \mathbb{Q}$ into \eqref{eqn:GM_CDT_theta} to retrieve \eqref{eqn:can_be_improved}. To see this, recall that $n_{\L} = \#\mathcal{G} = \varphi(q)$, $\#\mathcal{C} = 1$, and 
\begin{equation*}
    \Delta_{\mathbb{L}} = (-1)^{\tfrac{\varphi(q)}{2}} q^{\varphi(q)} \prod_{p|q} p^{-\tfrac{\varphi(q)}{p-1}} .
\end{equation*}
Likewise, insert $\L = \K = \Q$ into \eqref{eqn:GM_CDT_theta} to retrieve \eqref{eqn:pnts_smallT}.
\end{proof}

Further to the explicit result for $\theta(x)$ that was established in \eqref{eqn:pnts_smallT}, the following result from Schoenfeld does a better job when $x\geq 599$.

\begin{theorem}[\text{Schoenfeld \cite[Theorem~10]{Schoenfeld}}]\label{thm:pnts_bigT}
If the RH is true and $x\geq 599$, then
\begin{equation*}
    |\theta(x) - x| \leq \frac{\sqrt{x}(\log{x})^2}{8\pi} .
\end{equation*}
\end{theorem}

\begin{remark}
When $x$ is large, improvements over \eqref{eqn:can_be_improved} have been obtained in \cite[Corollary~1.1]{LeePNTPAP}. We have strived to keep the number of conditions on $x$ to a minimum, so we favored this form for this paper. Other explicit versions of the Prime Number Theorem could be applied with only minor variations to what we are presenting here. 
\end{remark}

\subsection{Technical Lemmas for the Proof of Theorem \ref{thm:explicit_fmla_constant}}\label{ssec:TL}

To prove Theorem \ref{thm:explicit_fmla_constant}, we will require two technical lemmas. The first is given below.

\begin{lemma}\label{lem:need_l8er}
If $x>1$, then
\begin{equation*}
    \left| \sum_{p > x} \frac{1}{p(p-1)} \right|
    < \frac{1}{x-1} \left(2 B + \frac{1}{\log{x}} + \frac{2}{(\log{x})^2}\right) .
\end{equation*}
\end{lemma}

\begin{proof}
Partial summation tells us
\begin{equation*}
    \sum_{p > x} \frac{1}{p(p-1)} = \frac{M_2(x)}{1-x} + \int_x^\infty \frac{M_2(t)}{(t-1)^2}\,dt,
\end{equation*}
in which
\begin{equation}\label{eqn:rosser_bound}
    M_2(x) = \sum_{p\leq x} \frac{1}{p} < u(x) := \log\log{x} + B + \frac{1}{(\log{x})^2}
    \quad\text{for}\quad x>1,
\end{equation}
using \cite[(3.20)]{Rosser}. It follows, using \eqref{eqn:rosser_bound} and integration by parts, that
\begin{align*}
    &\left| \sum_{p > x} \frac{1}{p(p-1)} \right|
    < \frac{u(x)}{x-1} + \int_x^\infty \frac{\log\log{t}}{(t-1)^2}\,dt + \left(B + \frac{1}{(\log{x})^2}\right)\int_x^\infty \frac{dt}{(t-1)^2} \\
    &\qquad\qquad=\frac{u(x)}{x-1} - \frac{\log\log{x}}{x-1} + \int_x^\infty \frac{t^{-1}dt}{(t-1) \log{t}} + \left(B + \frac{1}{(\log{x})^2}\right)\int_x^\infty \frac{dt}{(t-1)^2} \nonumber\\
    &\qquad\qquad\leq \frac{u(x) - \log\log{x}}{x-1} + \left(B + \frac{1}{\log{x}} + \frac{1}{(\log{x})^2}\right)\int_x^\infty \frac{dt}{(t-1)^2} \\
    &\qquad\qquad= \frac{1}{x-1} \left(2 B + \frac{1}{\log{x}} + \frac{2}{(\log{x})^2}\right). \qedhere
\end{align*}
\end{proof}

Next, the Prime Number Theorem for primes in arithmetic progression tells us to expect
\begin{equation*}
    \sum_{\substack{p> T \\ p\,\equiv\,\ell\imod{q}}} \frac{1}{p} \sim \frac{1}{\varphi(q)}\sum_{p> T} \frac{1}{p} 
    \qquad\text{for}\qquad T\geq 3 ,
\end{equation*}
but we will need an effective description of the error in this relationship later. To this end, we prove the next technical lemma, which relies on three applications of the explicit Prime Number Theorems that were presented in Section \ref{ssec:EPNT}. 

\begin{lemma}\label{lem:computer}
Suppose the GRH is true and set
\begin{equation*}
\begin{split}
    &f_1(q) :=  \frac{15 + \frac{3}{\varphi(q)}}{8\pi}, \quad f_2(q) :=\frac{6 + \frac{15}{4\pi} + \frac{3}{4\pi\varphi(q)}}{\log{q}}, \quad f_3(q) :=  \frac{22.29}{(\log{q})^2},\\ &f_4(q) := \frac{20.58}{(\log{q})^3},\quad
    f_5(q) := \frac{3.43 + \frac{1}{4\pi}}{\varphi(q) (\log{q})^2},\quad \text{and} \quad  f_6(q) := \frac{3}{4\pi \varphi(q)(\log{q})^3}.
\end{split}
\end{equation*}
If $q \geq 3$, then
\begin{equation*}    \Bigg|\sum_{\substack{p\,\equiv\,\ell\imod{q}\\p > q}} \frac{1}{p} - \frac{1}{\varphi(q)} \sum_{p > q} \frac{1}{p}\Bigg|\leq \frac{\log{q}}{\sqrt{q}} \sum_{i=1}^6 f_i(q) .
\end{equation*}
If $q \geq 599$, then
\begin{equation*}
    \Bigg|\sum_{\substack{p\,\equiv\,\ell\imod{q}\\p > q}} \frac{1}{p} - \frac{1}{\varphi(q)} \sum_{p > q} \frac{1}{p}\Bigg|\leq \frac{\log{q}}{\sqrt{q}} \sum_{i=1}^4 f_i(q) .
\end{equation*}
\end{lemma}

\begin{proof}
Suppose that $h(t) = \frac{1}{t\log{t}}$. It is clear that $h(t) > 0$ and $h'(t) < 0$, so it follows from Theorem \ref{thm:pnts_bigT} that if $T\geq 599$, then 
\begin{align}
    \sum_{p> T} &\,\frac{1}{p} 
    = - \theta(T) h(T) - \int_T^\infty \theta(t) h'(t) \,dt \nonumber\\
    &= - T h(T) - \int_T^\infty t h'(t) \,dt + O^\star\left(\frac{h(T) \sqrt{T} (\log{T})^2}{8\pi} - \int_T^\infty \frac{h'(t) \sqrt{t} (\log t)^2}{8\pi}\,dt\right) \nonumber\\
    &= - \frac{1}{\log{T}} - \int_T^\infty t h'(t) \,dt + O^\star\left(\frac{\log{T}}{8\pi \sqrt{T}} + \frac{\log{T} + 3}{4\pi \sqrt{T}}\right) \nonumber\\
    &= - \frac{1}{\log{T}} - \int_T^\infty t h'(t) \,dt + O^\star\!\left(\left(\frac{3}{8\pi} + \frac{3}{4\pi\log{T}}\right) \frac{\log{T}}{\sqrt{T}} \right) . \label{eqn:moi1}
\end{align}
Likewise, it follows from \eqref{eqn:pnts_smallT} in Theorem \ref{thm:pnts} that if $T\geq 3$, then 
\begin{align}
    \sum_{p> T} &\,\frac{1}{p} + T h(T) + \int_T^\infty t h'(t) \,dt \nonumber\\
    &= O^\star\left(\frac{h(T) \sqrt{T} (\log{T})^2}{8\pi} + 3.43 h(T) \sqrt{T} - \int_T^\infty h'(t) \sqrt{t} \left(\frac{(\log t)^2}{8\pi} + 3.43\right) \,dt\right) \nonumber\\
    &= O^\star\left(\frac{\log{T} + \frac{3.43}{\log{T}}}{8\pi \sqrt{T}} + \left(1 + \frac{1}{(\log{T})^2}\right)\frac{\log{T} + 3}{4\pi \sqrt{T}}\right) \nonumber\\
    &= O^\star\!\left(\left(\frac{3}{8\pi} + \frac{3}{4\pi \log{T}} + \frac{3.43 + \frac{1}{4\pi}}{(\log{T})^2} + \frac{3}{4\pi (\log{T})^3}\right)\frac{\log{T}}{\sqrt{T}} \right) . \label{eqn:moi2}
\end{align}
Furthermore, it follows from \eqref{eqn:can_be_improved} in Theorem \ref{thm:pnts} that if $T\geq 3$, then
\begin{align}
    \sum_{\substack{p> T \\ p \equiv \ell \imod{q}}} &\frac{1}{p} + \frac{T h(T)}{\varphi(q)} + \int_T^\infty \frac{t h'(t)}{\varphi(q)}\,dt \nonumber\\
    &= O^\star\left( \left(\frac{(\log{T})^2}{8\pi} + \left(\frac{\log{T}}{2\pi} + 2\right)\log{q} + 3.43\right) \sqrt{T} h(T) \right.\nonumber\\
    &\qquad\quad\quad \left. - \int_T^\infty \left(\frac{(\log{t})^2}{8\pi} + \left(\frac{\log{t}}{2\pi} + 2\right)\log{q} + 3.43\right) \sqrt{t} h'(t) \,dt \right) \nonumber\\
    &= O^\star\left( \left(\frac{(\log{T})^2}{8\pi} + \left(\frac{\log{T}}{2\pi} + 2\right)\log{q} + 3.43\right) \frac{1}{\sqrt{T}\log{T}} \right.\nonumber\\
    &\qquad\quad\quad \left. + \left(\frac{1}{4\pi} + \frac{6.86}{(\log{T})^2} + \left(\frac{1}{\pi\log{T}} + \frac{4}{(\log{T})^2}\right)\log{q}\right) \frac{\log{T} + 3}{\sqrt{T}} \right) \nonumber\\
    &= O^\star\!\left(\frac{\eta_1(T) \log{T} + \eta_2(T) \log{q}}{\sqrt{T}} \right) , \label{eqn:moi3}
\end{align}
in which
\begin{align*}
    \eta_1(T) &= \frac{3}{8\pi} + \frac{3}{4\pi\log{T}} + \frac{10.29}{(\log{T})^2} + \frac{20.58}{(\log{T})^3} , \\
    \eta_2(T) &= \frac{3}{2\pi} + \frac{6 + 3\pi^{-1}}{\log{T}} + \frac{12}{(\log{T})^2} .
\end{align*}
Finally, \eqref{eqn:moi1}, \eqref{eqn:moi2}, and \eqref{eqn:moi3} tell us that if $q\geq 3$ and the GRH is true, then
\begin{align*}
    &\Bigg|\sum_{\substack{p\,\equiv\,\ell\imod{q}\\p > q}} \frac{1}{p} - \frac{1}{\varphi(q)} \sum_{p > q} \frac{1}{p}\Bigg| \\
    &\,\,\,\leq 
    \begin{cases}
        \displaystyle \left(\left(\frac{3}{8\pi} + \frac{3}{4\pi \log{q}} \right)\frac{1}{\varphi(q)} + f_5(q) + f_6(q) + \eta_1(q) + \eta_2(q)\right) \frac{\log{q}}{\sqrt{q}} &\text{if }q\geq 3,\\
        ~\\
        \displaystyle \left(\left(\frac{3}{8\pi} + \frac{3}{4\pi \log{q}}\right)\frac{1}{\varphi(q)} + \eta_1(q) + \eta_2(q)\right) \frac{\log{q}}{\sqrt{q}} &\text{if }q\geq 599,
    \end{cases} 
\end{align*}
and the result follows.
\end{proof}

\subsection{Lower Bounds for the Euler Totient-Function}\label{ssec:varphi_lower_bounds}

For the purpose of our final computations, we will require explicit lower bounds for $\varphi(n)$, a selection of which are presented in the following lemma.

\begin{lemma}\label{lem:varphi_lower_bounds}
If $n\geq 3$, then 
\begin{equation*}
    \varphi(n) > 
    \begin{cases}
        \displaystyle \frac{\log{2}}{2} \frac{n}{\log{n}} &\text{if $n\geq 3$,}\\
        \displaystyle \sqrt{n} &\text{if $n\neq 3$ and $n\neq 6$,}\\
        \displaystyle n^{2/3} &\text{if $n > 30$.}
    \end{cases}
\end{equation*}
\end{lemma}

\begin{proof}
Hatalov\'{a} and \v{S}al\'{a}t proved the result for $n\geq 3$ in \cite{HatlovaSalat}. 
Vaidya proved the result for $n\neq 3$ and $n\neq 6$ in \cite{Vaidya}.
Kendall and Osborn proved the result for $n > 30$ in \cite{KendallOsborn}. 
These references were found in \cite{HandbookNTI}.
\end{proof}

\section{Proof of Main Results}\label{sec:main_theorems}

In this section, we prove the main results of this paper, bringing forward all previously established notations. 

\subsection{Proof of Theorem \ref{thm:explicit_fmla_constant}}

To begin, insert \eqref{eq:C_LZ} into the definition of $M(q,\ell)$ in \eqref{eqn:MertensSecondAPs} to see that
\begin{align}
    M(q,\ell) &= - \log{C(q,\ell)} +  B(q,\ell) \nonumber\\
    &= \frac{\gamma}{\varphi(q)} - \sum_p \frac{\alpha(p;q,\ell)}{\varphi(q)} \log\left(1 - \frac{1}{p}\right) + \sum_{p\,\equiv\,\ell\imod{q}} \left\{\frac{1}{p} + \log\left(1 - \frac{1}{p}\right)\right\} \nonumber\\
    &= \frac{\gamma}{\varphi(q)} + \sum_{p\,\equiv\,\ell\imod{q}} \left\{\frac{1}{p} + \left(1 - \frac{\varphi(q) - 1}{\varphi(q)}\right)\log\left(1 - \frac{1}{p}\right)\right\} \nonumber\\
    &\hspace{6cm}+ \sum_{p\,\not\equiv\,\ell\imod{q}} \frac{1}{\varphi(q)} \log\left(1 - \frac{1}{p}\right) \nonumber\\
    &= \frac{\gamma}{\varphi(q)} + \sum_{p} \left\{\frac{\beta(p;q,\ell)}{p} + \frac{1}{\varphi(q)} \log\left(1 - \frac{1}{p}\right)\right\} . \label{eqn:Step1}
\end{align}
Re-write the sum over primes in \eqref{eqn:Step1} as
\begin{equation*}
    \sum_{p} \left\{\frac{\beta(p;q,\ell)}{p} + \frac{1}{\varphi(q)} \log\left(1 - \frac{1}{p}\right)\right\}
    = \mathcal{A}(q) + \mathcal{B}(q) ,
\end{equation*}
in which
\begin{align*}
    \mathcal{A}(q) &= \sum_{p \leq q} \left\{\frac{\beta(p;q,\ell)}{p} + \frac{1}{\varphi(q)} \log\left(1 - \frac{1}{p}\right)\right\} , \\
    \mathcal{B}(q) &= \sum_{p > q} \left\{\frac{\beta(p;q,\ell)}{p} + \frac{1}{\varphi(q)} \log\left(1 - \frac{1}{p}\right)\right\} .
\end{align*}
We will bound $\mathcal{A}(q)$ and $\mathcal{B}(q)$ in the following lemmas. The result on $\mathcal{A}(q)$ is unconditional, whereas the bound on $\mathcal{B}(q)$ depends on the GRH.

\begin{lemma}\label{lem:Aq}
If $q\geq 2$, then
\begin{equation*}
    \mathcal{A}(q) = \frac{\delta_{\ell}}{\ell} - \frac{\gamma + \log\log{q}}{\varphi(q)} + O^\star\!\left(\frac{\varphi(q)^{-1}}{(\log{q})^2}\right) .
\end{equation*}
\end{lemma}

\begin{proof}
Recall that \cite[(3.25)-(3.27)]{Rosser} tell us that
\begin{equation*}
    \left| \prod_{p\leq q}\left(1 - \frac{1}{p}\right) - \frac{e^{-\gamma}}{\log{q}} \right| \leq
    \begin{cases}
        \frac{e^{-\gamma}}{2(\log{q})^3} & \text{if } q\geq 285 , \\
        \frac{e^{-\gamma}}{(\log{q})^3} & \text{if } q\geq 2 .
    \end{cases}
\end{equation*}
It follows that 
\begin{align*}
    \mathcal{A}(q) 
    &= \sum_{\substack{p \leq q\\p\,\equiv\,\ell\imod{q}}} \frac{1}{p} + \frac{1}{\varphi(q)} \log\prod_{p \leq q} \left(1 - \frac{1}{p}\right) \\
    &= \sum_{\substack{p \leq q\\p\,\equiv\,\ell\imod{q}}} \frac{1}{p} + \frac{1}{\varphi(q)} \log\!\left( \frac{e^{-\gamma}}{\log{q}} \left(1 + O^\star\!\left(\frac{1}{(\log{q})^2}\right)\right)\right) \\
    &= \sum_{\substack{p \leq q\\p\,\equiv\,\ell\imod{q}}} \frac{1}{p} - \frac{\gamma + \log\log{q}}{\varphi(q)} + \frac{1}{\varphi(q)}\log\!\left(1 + O^\star\!\left(\frac{1}{(\log{q})^2}\right)\right) \\
    &= \sum_{\substack{p \leq q\\p\,\equiv\,\ell\imod{q}}} \frac{1}{p} - \frac{\gamma + \log\log{q}}{\varphi(q)} + O^\star\!\left(\frac{\varphi(q)^{-1}}{(\log{q})^2}\right) .
\end{align*}
Moreover, $p\leq q$ is congruent to $\ell\imod{q}$ if and only if $\ell$ is also prime, so the result follows naturally.
\end{proof}

\begin{lemma}\label{lem:Bq}
Assume the GRH and recall the definitions of $f_i(q)$ for $1\leq i\leq 6$ that were introduced in Lemma \ref{lem:computer}. If $q\geq 3$, then
\begin{equation*}
    |\mathcal{B}(q)|
    \leq \tau_1(q) \frac{\varphi(q)^{-1}}{q-1} + \frac{\log{q}}{\sqrt{q}} \sum_{i=1}^6 f_i(q) .
\end{equation*}
If $q \geq 599$, then 
\begin{equation*}
    |\mathcal{B}(q)|
    \leq \tau_1(q) \frac{\varphi(q)^{-1}}{q-1} + \frac{\log{q}}{\sqrt{q}} \sum_{i=1}^4 f_i(q) .
\end{equation*}
\end{lemma}

\begin{proof} 
\red{Note that
\begin{align*}
    \frac{1}{\varphi(q)} \sum_{p > q} \log\left(1 - \frac{1}{p}\right)
    &= - \frac{1}{\varphi(q)} \sum_{p > q} \sum_{m\geq 1} \frac{1}{mp^m} \\
    &= - \frac{1}{\varphi(q)} \left(\sum_{p > q} \frac{1}{p} + O^*\left(\frac{1}{2} \sum_{p > q} \sum_{m\geq 2} \frac{1}{p^m} \right)\right) \\
    &= - \frac{1}{\varphi(q)} \left(\sum_{p > q} \frac{1}{p} + O^*\left(\frac{1}{2} \sum_{p > q} \frac{1}{p(p - 1)} \right)\right)
\end{align*}}
Therefore, Lemma \ref{lem:need_l8er} implies that if $q\geq 3$ and the GRH is true, then
\begin{align*}
    |\mathcal{B}(q)|
    &\leq \red{\Bigg|\sum_{\substack{p\,\equiv\,\ell\imod{q}\\p > q}} \frac{1}{p} - \frac{1}{\varphi(q)} \sum_{p > q} \frac{1}{p}\Bigg| + \frac{1}{2 \varphi(q)} \sum_{p > q} \frac{1}{p(p-1)}} \\
    &\leq \Bigg|\sum_{\substack{p\,\equiv\,\ell\imod{q}\\p > q}} \frac{1}{p} - \frac{1}{\varphi(q)} \sum_{p > q} \frac{1}{p} \Bigg| + \red{\left(B + \frac{1}{2\log{q}} + \frac{1}{(\log{q})^2}\right)} \frac{\varphi(q)^{-1}}{q-1} \\
    &= \Bigg|\sum_{\substack{p\,\equiv\,\ell\imod{q}\\p > q}} \frac{1}{p} - \frac{1}{\varphi(q)} \sum_{p > q} \frac{1}{p} \Bigg| + \tau_1(q) \frac{\varphi(q)^{-1}}{q-1} .
\end{align*}
Furthermore, Lemma \ref{lem:computer} implies the result.
\end{proof}

Insert Lemma \ref{lem:Aq} into \eqref{eqn:Step1} to see (unconditionally) that
\begin{equation}\label{eqn:Step2}
    \left|M(q,\ell) - \frac{\delta_{\ell}}{\ell} + \frac{\log\log{q}}{\varphi(q)} - \mathcal{B}(q) \right| \leq \frac{\varphi(q)^{-1}}{(\log{q})^2} ,
\end{equation}
so all that remains to describe $M(q,\ell)$ is to control the size of the contribution from $\mathcal{B}(q)$; this is an explicit version of \eqref{eqn:look_at_me}. To this end, insert Lemma \ref{lem:Bq} into \eqref{eqn:Step2} to see that if $q\geq 3$ and the GRH is true, then
\begin{equation*}
    \left| M(q,\ell) - \frac{\delta_{\ell}}{\ell} + \frac{\log\log{q}}{\varphi(q)}\right|
    \leq \tau_1(q) \frac{\varphi(q)^{-1}}{q-1} + \eta(q) \frac{\log{q}}{\sqrt{q}} ,
\end{equation*}
in which $\eta(q)$ is equal to
\begin{equation*}
    \begin{cases}
        \displaystyle \frac{15 + \frac{3}{\varphi(q)}}{8\pi} + \frac{6 + \frac{15}{4\pi} + \frac{3}{4\pi\varphi(q)}}{\log{q}} + \frac{22.29 + \frac{4.43 + \frac{1}{4\pi}}{\varphi(q)}}{(\log{q})^2} + \frac{20.58 + \frac{3}{4\pi\varphi(q)}}{(\log{q})^3} 
        &\text{if }q\geq 3,\\
        ~\\
        \displaystyle \frac{15 + \frac{3}{\varphi(q)}}{8\pi} + \frac{6 + \frac{15}{4\pi} + \frac{3}{4\pi\varphi(q)}}{\log{q}} + \frac{22.29 + \varphi(q)^{-1}}{(\log{q})^2} + \frac{20.58}{(\log{q})^3}
        &\text{if }q\geq 599 .
    \end{cases}
\end{equation*}
Using the lower bounds for $\varphi(q)$ from Lemma \ref{lem:varphi_lower_bounds}, we are able to replace any occurrence of $\varphi(q)$ with a lower bound that will enable us to replace $\eta(q)$ with a decreasing function in $q$. Now, of the bounds in Lemma \ref{lem:varphi_lower_bounds}, the last is the strongest for $n\leq 24\,924$, and the first is stronger otherwise. Therefore, we note that if $q\geq 10\,000$ and the GRH is true, then
\begin{equation*}
    \eta(q) < \frac{15 + \frac{3}{q^{2/3}}}{8\pi} + \frac{6 + \frac{15}{4\pi} + \frac{3}{4\pi q^{2/3}}}{\log{q}} + \frac{22.29 + q^{- 2/3}}{(\log{q})^2} + \frac{20.58}{(\log{q})^3} .
\end{equation*}
Next, computations show that the same upper bound also holds for all $q\geq 210$. Therefore, to capture a bound that holds for all $q\geq 3$, we alter the form of this upper bound and use computations to find that if $q\geq 3$ and the GRH is true, then
\begin{equation*}
    \eta(q) 
    < \frac{15 + \frac{84.89}{q^{2/3}}}{8\pi} + \frac{6 + \frac{15}{4\pi} + \frac{3}{4\pi q^{2/3}}}{\log{q}} + \frac{22.29 + q^{- 2/3}}{(\log{q})^2} + \frac{20.58}{(\log{q})^3}
    = \tau_2(q) .
\end{equation*}
To find the constant $84.89$, we incremented the original constant $3$ by $0.01$ until the constant was large enough to hold computationally for all $q\geq 3$. The result (Theorem \ref{thm:explicit_fmla_constant}) follows naturally.

\subsection{Proof of Corollary \ref{cor:Norton_explicit}}

It follows from Theorem \ref{thm:explicit_fmla_constant} and the computations in Table \ref{tab:comps} that if $q\geq 3$ and the GRH is true, then
\begin{equation*}
    \left|M(q,\ell) - \frac{\delta_{\ell}}{\ell}\right| 
    < \frac{\log\log{q}}{\varphi(q)} + \frac{\red{1.54515} \varphi(q)^{-1}}{q-1} + \frac{43.26016 \log{q}}{\sqrt{q}} .
\end{equation*}
It follows that 
\begin{equation}\label{eqn:for_checks}
    \left|M(q,\ell) - \frac{\delta_{\ell}}{\ell}\right| 
    < \left(43.26016 + \frac{\sqrt{q} \log\log{q}}{\varphi(q) \log{q}} + \frac{\red{1.54515}  \sqrt{q}}{\varphi(q) (q-1) \log{q}}\right) \frac{\log{q}}{\sqrt{q}} .
\end{equation}
Alternatively, insert the last bound from Lemma \ref{lem:varphi_lower_bounds} into the initial relationship to see that if $q\geq 31$, then 
\begin{align*}
    \left|M(q,\ell) - \frac{\delta_{\ell}}{\ell}\right| 
    &< \frac{\log\log{q}}{q^{2/3}} + \frac{\red{1.54515}}{q^{2/3}(q-1)}  + \frac{43.26016 \log{q}}{\sqrt{q}} \\ 
    &= \left(43.26016 + \frac{\log\log{q}}{q^{1/6}\log{q}} + \frac{\red{1.54515}}{q^{1/6}(q-1)\log{q}} \right)\frac{\log{q}}{\sqrt{q}} \\ 
    &< \frac{\red{43.47133} \log{q}}{\sqrt{q}} , 
\end{align*}
since the last coefficient of $\log{q}/\sqrt{q}$ decreases on $q\geq 9$. Finally, using computations, we can observe that the largest coefficient of $\log{q}/\sqrt{q}$ in \eqref{eqn:for_checks} such that $q\in\{3,4,\ldots,30\}$ is \red{$43.94332$}. The result follows naturally.

\subsection*{Acknowledgements} 
ESL thanks the Heilbronn Institute for Mathematical Research for their support. We also thank Joshua Stucky for valuable feedback and bringing \cite{Pomerance} to our attention. Finally, we thank Alessandro Languasco, Kenneth Williams, and an anonymous referee for a careful reading of earlier drafts. 

\end{document}